\documentclass[11pt,reqno,english,a4]{amsart}
\usepackage{amsmath,amsthm,amsfonts,amssymb,amscd}
\usepackage[latin1]{inputenc}
\usepackage{psfrag}
\usepackage{epsfig}
\usepackage{a4wide}
\usepackage[]{graphicx}
\usepackage{times}
\usepackage{bm}
\usepackage{mathdots}
\usepackage{mathrsfs}
\usepackage{caption}

\newcommand{\R}{\mathbb{R}}

\newcommand{\fg}{\mathfrak g}
\newcommand{\fa}{\mathfrak a}
\newcommand{\fh}{\mathfrak h}
\newcommand{\fk}{\mathfrak k}

\newcommand{\fs}{\mathfrak {so}}
\newcommand{\fn}{\mathfrak n}

\newcommand{\Iso}{\mathrm{Iso}}
\newtheorem{definition}{Definition}[section]
\newtheorem{lemma}[definition]{Lemma}
\newtheorem{remark}[definition]{Remark}
\newtheorem{theorem}[definition]{Theorem}

\newtheorem{rem}[definition]{Remark}
\newtheorem{pro}[definition]{Proposition}

\def\co {cohomogeneity one}

\begin{document}
	
	\title[A classification of cohomogeneity one actions on the Minkowski space $\mathbb{R}^{3,1}$]
	{A classification of cohomogeneity one actions on the Minkowski space $\mathbb{R}^{3,1}$}
	
	\author{P. Ahmadi}	
	\author{S. Safari}
	\author{M. Hassani}

	\thanks{}
	
	
	
	\date{\today}
	\address{
		P. Ahmadi\\
		Departmental of mathematics\\
		University of Zanjan\\
		University blvd.\\
		Zanjan\\
		Iran}
	\email{p.ahmadi@znu.ac.ir}
	
	\address{
		S. Safari\\
		Departmental of mathematics\\
		University of Zanjan\\
		University blvd.\\
		Zanjan\\
		Iran}
	\email{salim.safari@znu.ac.ir}
	
	\address{
			M. Hassani\\
			Departmental of mathematics\\
			University of Zanjan\\
			University blvd.\\
			Zanjan\\
			Iran}
		\email{masoud.hasani@znu.ac.ir}

	\begin{abstract}
		The aim of this paper is to classify cohomogeneity one isometric actions on the $4$-dimensional Minkowski space $\R^{3,1}$, up to orbit equivalence. Representations, up to conjugacy, of the acting groups in $O(3,1)\ltimes \R^{3,1}$ are given in both cases, proper and non-proper actions. When the action is proper, the orbits and the orbit spaces are determined.
	\end{abstract}
	
	\maketitle
	\tableofcontents
	\medskip
	
	\thispagestyle{empty}
	

\section{Introduction}
An action of a Lie group $G$ on a connected manifold $M$ is called of cohomogeneity one, if the minimum codimension of the induced $G$-orbits in $M$ is one. The concept of a cohomogeneity one action on a manifold M was introduced by P.S. Mostert in his 1956 paper \cite{Mos}, wherein it is assumed that the acting group is compact. With this assumption, he determined the orbit space up to homeomorphism. More precisely, he proved that by the cohomogeneity one action of a compact Lie group $G$ on a manifold $M$ the orbit space $M/G$ is homeomorphic to one of the spaces $\R$, $\mathbb{S}^1$, $[0,1]$, or $[0,1)$. For the general case, in \cite{Br}, B. Bergery showed that
if $M$ is a Riemannian manifold and $G$, a closed Lie subgroup of $\Iso(M)$, acts isometrically and with cohomogeneity one on $M$, then the orbit space is one of the mentioned spaces. The key hypothesis was the closeness of the acting group in $\Iso(M)$. 

An action of a Lie group $G$ on a manifold $M$ is said to be {\it proper} if the mapping $\varphi:G\times M\rightarrow M\times M , \ (g,x)\mapsto(g.x,x)$ is
proper. A result by D. Alekseevsky in \cite{Alee} says that, the action of $G$ on $M$ is proper if and only if there is a complete $G$-invariant Riemannian metric on $M$ such that $G$ is closed in $\Iso(M)$. This theorem provides a link between proper actions and Riemannian $G$-manifolds.

Cohomogeneity one Riemannian manifolds have been studied by many mathematicians (see, e.g., \cite{AA, Br, GWZ, GZ1, GZ2, MK, Mos, PT, PS, S, V1, V2}). The common hypothesis in the theory is that the acting group is closed in the full isometry group of the Riemannian manifold and the action is isometric. When the metric on $M$ is indefinite, this assumption in general does not imply that the action is proper, so the study becomes much more complicated. 

The most natural way to study a cohomogeneity one semi-Riemannian manifold $M$ is to determine the acting group in $\Iso(M)$, up to conjugacy, since the actions of two conjugate subgroups in $\Iso(M)$ induce almost the same orbits in $M$. This has been done for space forms in some special cases (see \cite{A1, A2, AhA, Ahmadi}). This way is the one that we pursue in this paper.

Here, we assume that $M$ is the simplest example of a relativistic spacetime, the $4$-dimensional Minkowski space $\R^{3,1}$, and $G$ is a connected Lie subgroup of $\Iso(\R^{3,1})$ which acts on $\R^{3,1}$ isometrically and with cohomogeneity one. We give explicit representations of such groups in $\Iso(\R^{3,1})=O(3,1)\ltimes \R^{3,1}$, up to conjugacy. Then we determine those acting properly and nonproperly. When the action is proper, we specify the induced orbits and the orbit spaces.

\section{Preliminaries}\label{Lie algebra}

Let $G$ be a Lie group which acts on a connected smooth manifold $M$. The action is called of {\it cohomogeneity one}, if the minimum codimension among the induced $G$-orbits in $M$ is one.  For each point $x$ in $M$, $G(x)$ denotes the orbit of $x$, and $G_x$ is the stabilizer in $G$ of $x$. The action is said to be {\it proper} if the mapping
$\varphi:G\times M\rightarrow M\times M , \ (g,x)\mapsto(g.x,x)$ is
proper. Equivalently, for any sequences $x_n$ in $M$ and $g_n$ in $G$, $g_nx_n\rightarrow y$ and $x_n\rightarrow x$ imply that $g_n$ has a convergent subsequence. The $G$-action on $M$ is {\it nonproper} if it is not proper. Equivalently, there are sequences $g_n$ in $G$ and $x_n$ in $M$ such that $x_n$ and $g_nx_n$ converge in $M$ and $g_n\rightarrow \infty$, i.e. $g_n$ leaves compact subsets. For instance, if $G$ is compact, the action is obviously proper. The orbit
space $M/G$ of a proper action of $G$ on $M$ is Hausdorff, the
orbits are closed submanifolds, and the stabilizers are
compact (see \cite{Adams}).

The Minkowski $4$-space $\R^{3,1}$ is the $4$-dimensional real vector space $\R^{4}$ with the line element $ds^{2}=dx_{1}^{2}+dx_{2}^{2}+dx_{3}^{2}-dx_{4}^{2}$. The orthogonal group $O(3,1)$ of the scalar product obtained of this line element is known as the Lorentz group and its elements are called Lorentz transformations of $\R^{3,1}$. The isometry group $\Iso (\R^{3,1})$ is the semidirect product $O(3,1)\ltimes_\tau \R^{3,1}$ with $\tau: O(3,1)\times \R^{3,1}\rightarrow \R^{3,1}$, $(V,v)\mapsto V(v)=Vv$. The multiplication and inversion on $\Iso (\R^{3,1})$ is given by $(V,v)(U,u)=(VU,v+V(u))$ and $(V,v)^{-1}=(V^{-1},-V^{-1}(v))$ and the action of $\Iso(\R^{3,1})$ on $\R^{3,1}$ is given by $\Iso(\R^{3,1})\times \R^{3,1}\rightarrow \R^{3,1}$, $((V,v),p)\mapsto V(p)+v$. The isometry group $\Iso(\R^{3,1})$ has four connected components. The identity component of $\Iso(\R^{3,1})$ is denoted by $\Iso_{\circ}(\R^{3,1})=SO_{\circ}(3,1)\ltimes \R^{3,1}$.

The Lie algebra of $\Iso(\R^{3,1})$ is the semidirect sum $\mathfrak{so}(3,1)\oplus_{\varphi}\R^{3,1}$, where $\varphi:  \mathfrak{so}(3,1)\oplus \R^{3,1}\rightarrow \R^{3,1}$, is defined by $\varphi(X+x)=X(x)=Xx$. The Lie bracket on $\mathfrak{so}(3,1)\oplus_\varphi \R^{3,1}$ is given by 
\begin{equation}\label{braketrule}[X+x,Y+y]=(XY-YX)+(Xy-Yx),\end{equation}
which yields the adjoint representation as follows 
$$Ad(U,u)(X+x)=UXU^{-1}+(Ux-UXU^{-1}u).$$
The Lie algebra of $SO_\circ (3,1)$ is given by
\begin{align*}
\mathfrak{so}(3,1) = \left\lbrace  \left(\begin{array}{cc}C&c \\
c^{t} & 0 \end{array}\right): C\in \mathfrak{so}(3), c\in \R^3 \right\}.
\end{align*} 
Considering the Cartan involution $\theta (X)=-X^t$ on $\mathfrak{so}(3,1)$ one gets the Cartan decomposition $\mathfrak{so}(3,1)=\mathfrak{k}\oplus \mathfrak{p}$ where,
\begin{align*}
\fk = \left\lbrace  \left(\begin{array}{cc}C&0 \\
0 & 0 \end{array}\right)\in \fs (3,1)\right\}\cong \fs (3),\quad {\rm and} \quad
\mathfrak{p}=\left\lbrace  \left(\begin{array}{cc}0&c \\
c^{t}& 0 \end{array}\right)\in \fs (3,1) \right\}\cong \R^{3}.
\end{align*}
The subspace 
$$\fa =\R\left(\begin{array}{cc}0&e_4\\(e_4)^t&0\end{array}\right)$$
is a maximal abelian subspace of $\mathfrak{p}$, where $(e_1,e_2,e_3,e_4)$ is the standard orthonormal basis of $\R^{3,1}$.
Let $$\fn=\left\{\left(\begin{array}{ccc}0&c&c\\-c^t&0&0\\c^t&0&0\end{array}\right):\ c\in \R^2\right\}.$$
Then $\fs (3,1)=\fk\oplus\fa\oplus\fn$ is an Iwasawa decomposition of $\fs (3,1)$. We denote by $K=SO(3)$, $A$ and $N$
the connected closed subgroups of $SO_\circ(3,1)$
with Lie algebras $\fk$, $\fa$ and $\fn$, respectively. We use the corresponding Iwasawa decomposition $SO_\circ(3,1)=KAN$ throughout the paper.

\section{Groups acting with cohomogeneity one on $\R^{3,1}$}
In this section we classify cohomogeneity one actions on the four dimensional Minkowski space $\R^{3,1}$ up to orbit-equivalent. We first fix some notations. We define the light-like line $ \ell\subset \R^{3,1}$ by $\ell=\R(e_{3}- e_{4})$, the degenerate plane $\mathbb{W}^2\subset \R^{3,1}$ by $\mathbb{W}^2=\R e_2 \oplus \ell$ and the degenerate hyperplane $\mathbb{W}^3\subset \R^{3,1}$ by $\mathbb{W}^3=\R e_1 \oplus \mathbb{W}^2$. As described in Section \eqref{Lie algebra}, we denote by $SO_{\circ}(3,1)=KAN$ the Iwasawa decomposition of $SO_{\circ}(3,1)$ and by
$\mathfrak{so}(3,1)= \fk\oplus \fa\oplus \fn$ the Iwasawa decomposition of the Lie algebra $\mathfrak{so}(3,1)$. We define
\begin{align*}
&Y^{1}_{\fk}= E_{12}-E_{21},\quad  Y^{2}_{\fk}= E_{13}-E_{31},\quad Y^{3}_{\fk}= E_{23}-E_{32}\in \fk,\\
&Y_{\fa}=E_{34}+E_{43}\in \fa,\\
&Y^{1}_{\fn}=E_{13}+E_{14}-E_{31}+E_{41}, \quad
Y^{2}_{\fn}=E_{23}+E_{24}-E_{32}+E_{42}\in \fn. 
\end{align*}
where $ E_{ij} $ is the $4\times 4$ matrix whose $ (i,j)$-entry is $ 1 $ and other entires are all $ 0$. Then one gets that
\begin{align}\label{brackets31}
&[Y_{\fk}^{1},Y_{\fk}^{2}]=-Y^{3}_{\fk},\quad 	[Y_{\fk}^{1},Y_{\fk}^{3}]=Y^{2}_{\fk},\nonumber\\
&[Y_{\fk}^{2},Y_{\fk}^{3}]=-Y^{1}_{\fk},\nonumber\\
&[Y_{\fk}^{1},Y_{\fn}^{1}]=-Y^{2}_{\fn},\quad 	[Y_{\fk}^{1},Y_{\fn}^{2}]=Y^{1}_{\fn},\nonumber\\
&[Y_{\fk}^{2},Y_{\fn}^{1}]=-Y_{\fa},\quad 	[Y_{\fk}^{2},Y_{\fn}^{2}]=-Y^{1}_{\fk},\\
&[Y_{\fk}^{3},Y_{\fn}^{1}]=Y^{1}_{\fk},\quad 	[Y_{\fk}^{3},Y_{\fn}^{2}]=-Y_{\fa},\nonumber \\
&[Y_{\fk}^{1},Y_{\fa}]=0,\quad 	[Y_{\fk}^{2},Y_{\fa}]=Y^{1}_{\fn}-Y^{2}_{\fk},\quad\nonumber\\
&[Y_{\fk}^{3},Y_{\fa}]=Y^{2}_{\fn}-Y^{3}_{\fk}\nonumber\\
&[Y_{\fa},Y_{\fn}^{1}]=-Y^{1}_{\fn},\quad [Y_{\fa},Y_{\fn}^{2}]=-Y^{2}_{\fn},\nonumber\\
&[Y_{\fn}^{1},Y_{\fn}^{2}]=0.\nonumber
\end{align}
If $ Y_{\mathfrak{k}}^{1}+u$, $Y_{\mathfrak{k}}^{2}+v$, $Y_{\mathfrak{k}}^{3}+w$, $Y_{\mathfrak{a}}+x$, $Y_{\fn}^{1}+y$ and $Y_{\fn}^{2}+z$ is a basis for a subalgebra $\fh\subseteq \fg$, where $u, v, w, x, y, z\in\R^{1,3}$, then the following vectors should belong to $\fh$.
\begin{align}\label{Bracketv31}
&[Y_{\fk}^{1}+u,Y_{\fk}^{2}+v]=-Y^{3}_{\fk}+(Y_{\fk}^{1}v-Y_{\fk}^{2}u),\quad 	[Y_{\fk}^{1}+u,Y_{\fk}^{3}+w]=Y^{2}_{\fk}+(Y_{\fk}^{1}w-Y_{\fk}^{3}u),\nonumber\\
&[Y_{\fk}^{2}+v,Y_{\fk}^{3}+w]=-Y^{1}_{\fk}+(Y_{\fk}^{2}w-Y_{\fk}^{3}v),\nonumber\\
&[Y_{\fk}^{1}+u,Y_{\fn}^{1}+y]=-Y^{2}_{\fn}+(Y_{\fk}^{1}y-Y_{\fn}^{1}u),\quad 	[Y_{\fk}^{1}+u,Y_{\fn}^{2}+z]=Y^{1}_{\fn}+(Y_{\fk}^{1}z-Y_{\fn}^{2}u),\nonumber\\
&[Y_{\fk}^{2}+v,Y_{\fn}^{1}+y]=-Y_{\fa}+(Y_{\fk}^{2}y-Y_{\fn}^{1}v),\quad 	[Y_{\fk}^{2}+v,Y_{\fn}^{2}+z]=-Y^{1}_{\fk}+(Y_{\fk}^{2}z-Y_{\fn}^{2}v),\\
&[Y_{\fk}^{3}+w,Y_{\fn}^{1}+y]=Y^{1}_{\fk}+(Y_{\fk}^{3}y-Y_{\fn}^{1}w),\quad 	[Y_{\fk}^{3}+w,Y_{\fn}^{2}+z]=-Y_{\fa}+(Y_{\fk}^{3}z-Y_{\fn}^{2}w),\nonumber \\
&[Y_{\fk}^{1}+u,Y_{\fa}+x]=(Y_{\fk}^{1}x-Y_{\fa}u),\qquad\quad 	[Y_{\fk}^{2}+v,Y_{\fa}+x]=Y^{1}_{\fn}-Y^{2}_{\fk}+(Y_{\fk}^{2}x-Y_{\fa}v),\quad\nonumber\\
&[Y_{\fk}^{3}+w,Y_{\fa}+x]=Y^{2}_{\fn}-Y^{3}_{\fk}+(Y_{\fk}^{3}x-Y_{\fa}w)\nonumber\\
&[Y_{\fa}+x,Y_{\fn}^{1}+y]=-Y^{1}_{\fn}+(Y_{\fa}^{1}y-Y_{\fn}^{1}x),\quad [Y_{\fa}+x,Y_{\fn}^{2}+z]=-Y^{2}_{\fn}+(Y_{\fa}z-Y_{\fn}^{2}x),\nonumber\\
&[Y_{\fn}^{1}+y,Y_{\fn}^{2}+z]=(Y_{\fn}^{1}z-Y_{\fn}^{2}y).\nonumber
\end{align}
These imply that
\begin{align}\label{vectors31}
&Y_{\fk}^{1}v-Y_{\fk}^{2}u=-w,\quad Y_{\fk}^{1}w-Y_{\fk}^{3}u=v,\quad
Y_{\fk}^{2}w-Y_{\fk}^{3}v=-u,\nonumber\\
&Y_{\fk}^{1}y-Y_{\fn}^{1}u=-z,\quad 	Y_{\fk}^{1}z-Y_{\fn}^{2}u=y,\nonumber\\
&Y_{\fk}^{2}y-Y_{\fn}^{1}v=-x,\quad 	Y_{\fk}^{2}z-Y_{\fn}^{2}v=-u,\\
&Y_{\fk}^{3}y-Y_{\fn}^{1}w=u,\quad \	Y_{\fk}^{3}z-Y_{\fn}^{2}w=-x,\nonumber\\
&Y_{\fk}^{1}x-Y_{\fa}u=0,\quad \quad	
Y_{\fk}^{2}x-Y_{\fa}v=y-v,\quad 	
Y_{\fk}^{3}x-Y_{\fa}w=z-w,\nonumber\\
&Y_{\fa}^{1}y-Y_{\fn}^{1}x=-y,\quad 	
Y_{\fa}z-Y_{\fn}^{2}x=-z,\nonumber\\
&Y_{\fn}^{1}z-Y_{\fn}^{2}y=0.\nonumber	
\end{align}
For arbitrary $p=(p_1,p_2,p_3,p_4)^{t}\in \R^{3,1}$ we have
\begin{align*}
&Y^{1}_{\fk}p = (p_{2}, -p_{1}, 0, 0)^{t},\quad\quad\quad\ \ Y^{2}_{\fk}p = (p_{3}, 0, -p_{1}, 0)^{t},\quad Y^{3}_{\fk}p = (0, p_{3}, -p_{2}, 0)^{t}\\
&Y^{1}_{\fn}p = (p_{3}+p_{4}, 0, -p_{1}, p_{1})^{t},\quad Y^{2}_{\fn}p = (0, p_{3}+p_{4}, -p_{2}, p_{2})^{t},\\
&Y_{\fa}p = (0, 0, p_{4}, p_{3})^{t}.
\end{align*}
Hence the relations in \eqref{vectors31} imply that
\begin{align*}
&\left(\begin{array}{c} v_{2}-u_{3}\\-v_{1}\\u_{1}\\0
\end{array}\right)=-\left(\begin{array}{c} w_{1}\\w_{2}\\w_{3}\\w_{4}
\end{array}\right),\quad 
\left(\begin{array}{c} w_{2}\\-w_{1}-u_{3}\\u_{2}\\0
\end{array}\right)=\left(\begin{array}{c} v_{1}\\v_{2}\\v_{3}\\v_{4}
\end{array}\right),\quad
\left(\begin{array}{c} w_{3}\\-v_{3}\\-w_{1}+v_{2}\\0
\end{array}\right)=-\left(\begin{array}{c} u_{1}\\u_{2}\\u_{3}\\u_{4}
\end{array}\right),
\end{align*}
\begin{align*}
&\left(\begin{array}{c} y_{2}-u_{3}-u_{4}\\-y_{1}\\u_{1}\\-u_{1}
\end{array}\right)=-\left(\begin{array}{c} z_{1}\\z_{2}\\z_{3}\\z_{4}
\end{array}\right),\quad 
\left(\begin{array}{c} z_{2}\\-z_{1}-u_{3}-u_{4}\\u_{2}\\-u_{2}
\end{array}\right)=\left(\begin{array}{c} y_{1}\\y_{2}\\y_{3}\\y_{4}
\end{array}\right),\\
&\left(\begin{array}{c} y_{3}-v_{3}-v_{4}\\0\\-y_{1}+v_{1}\\-v_{1}
\end{array}\right)=-\left(\begin{array}{c} x_{1}\\x_{2}\\x_{3}\\x_{4}
\end{array}\right),\quad
\left(\begin{array}{c} z_{3}\\-v_{3}-v_{4}\\-z_{1}+v_{2}\\-v_{2}
\end{array}\right)=-\left(\begin{array}{c} u_{1}\\u_{2}\\u_{3}\\u_{4}
\end{array}\right),\\
&\left(\begin{array}{c} -w_{3}-w_{4}\\y_{3}\\-y_{2}+w_{1}\\-w_{1}
\end{array}\right)=\left(\begin{array}{c} u_{1}\\u_{2}\\u_{3}\\u_{4}
\end{array}\right),\quad
\left(\begin{array}{c} 0\\z_{3}-w_{3}-w_{4}\\-z_{2}+w_{2}\\-w_{2}
\end{array}\right)=-\left(\begin{array}{c} x_{1}\\x_{2}\\x_{3}\\x_{4}
\end{array}\right),
\end{align*}
\begin{align*}
&\left(\begin{array}{c} x_{2}\\-x_{1}\\-u_{4}\\-u_{3}
\end{array}\right)=\left(\begin{array}{c} 0\\0\\0\\0
\end{array}\right),\quad 
\left(\begin{array}{c} x_{3}\\0\\-x_{1}-v_{4}\\-v_{3}
\end{array}\right)=\left(\begin{array}{c} y_{1}-v_{1}\\y_{2}-v_{2}\\y_{3}-v_{3}\\y_{4}-v_{4}
\end{array}\right),\quad
\left(\begin{array}{c} 0\\x_{3}\\-x_{2}-w_{4}\\-w_{3}
\end{array}\right)=\left(\begin{array}{c} z_{1}-w_{1}\\z_{2}-w_{2}\\z_{3}-w_{3}\\z_{4}-w_{4}
\end{array}\right),\\\\
&\left(\begin{array}{c} -x_{3}-x_{4}\\0\\y_{4}+x_{1}\\y_{3}-x_{1}
\end{array}\right)=-\left(\begin{array}{c} y_{1}\\y_{2}\\y_{3}\\y_{4}
\end{array}\right),\quad 
\left(\begin{array}{c} 0\\-x_{3}-x_{4}\\z_{4}+x_{2}\\z_{3}-x_{2}
\end{array}\right)=-\left(\begin{array}{c} z_{1}\\z_{2}\\z_{3}\\z_{4}
\end{array}\right),\quad
\left(\begin{array}{c} z_{3}+z_{4}\\-y_{3}-y_{4}\\-z_{1}+y_{2}\\z_{1}-y_{2}
\end{array}\right)=\left(\begin{array}{c} 0\\0\\0\\0
\end{array}\right).
\end{align*}
Therefore
\begin{align}\label{vectors4}
&	u=(u_{1},u_{2},0,0)^{t},\quad v=(v_{1},0,u_{2},0)^{t},\qquad\qquad\ w=(0, v_{1},-u_{1}, 0)^{t}\nonumber\\ 
&x=(0,0,x_{3}, v_{1})^{t},\quad y=(x_{3}+v_{1},0,u_{2},-u_{2})^{t},\quad z=(0, x_{3}+v_{1},-u_{1},u_{1})^{t},
\end{align}
which are key relations in determining the list of groups acting with cohomogeneity one on $\R^{3,1}$ (see the proof of Theorem \ref{main}). The first step in classifying cohomogeneity one actions on $\R^{3,1}$ is to determine Lie subgroups of $SO_\circ(3,1)$ of each dimension, up to conjugacy. First we recall the following theorem from \cite{Scala}.

\begin{theorem}\label{scala-t}
	 Let $H$ be a connected Lie subgroup of
	$SO(n,1)$ and assume that the action of $H$ on the Minkowski space $\R^{n,1}$ is	irreducible. Then $H = SO_\circ(n,1)$.
\end{theorem}

\begin{lemma}\label{subgroups}
	There is no five dimensional Lie subgroup of $SO_\circ(3,1)$. The only four dimensional connected Lie subgroup of $SO_\circ(3,1)$ is conjugate to $K_1AN$. Any three dimensional connected Lie subgroup of $SO_\circ(3,1)$ is conjugate to one of the following Lie subgroups. 
	$$SO_\circ(2,1),\quad  SO(3), \quad \exp(\R(\lambda Y_{\fk}^1 +\mu Y_\fa))N,$$
	where $\lambda$ and $\mu$ are fixed real numbers, and at least one of them is nonzero.
\end{lemma}
\begin{proof} By Theorem \ref{scala-t}, the group $SO_\circ (3,1)$ is the only connected Lie subgroup of $SO_\circ(3,1)$ which acts irreducibly. Thanks to this theorem, every connected proper Lie subgroup of $SO_\circ (3,1)$  preserves a nontrivial linear subspace of $\R^{3,1}$. Every one dimensional space-like, time-like or light-like subspace of $\R^{3,1}$ is conjugate under $O(3,1)$ to $\R e_{1}$, $\R e_{4}$ or $\ell$, respectively. Let $H$ be a Lie subgroup of $SO_\circ(3,1)$  which acts on $\R^{3,1}$ isometrically and preserves a nontrivial linear subspace $\mathscr{L}$.
\begin{itemize}
	\item If $\mathscr{L}=\R e_{1}$, then $H\subseteq SO_\circ(2,1)$,
	\item If $\mathscr{L}=\R e_{4}$, then $H\subseteq SO(3)$,
	\item If $\mathscr{L}=\ell$, then $H \subseteq K_{1}AN$.
\end{itemize}
And any two dimensional space-like, Lorentzian or degenerate subspace of $\R^{3,1}$ is conjugate under $O(3,1)$ to $\R e_{1}\oplus \R e_{2}$, $\R e_{3}\oplus \R e_{4}$ or $\R e_{2}\oplus \ell$, respectively.
\begin{itemize}
	\item If $\mathscr{L}$ is either $\R e_{1}\oplus \R e_{2}$ or $\R e_{3}\oplus \R e_{4}$ , then $H\subseteq SO(2)\times SO_\circ(1,1)$,
	\item If $\mathscr{L}=\R e_{2}\oplus \ell$, then $H\subseteq AN_{1}$.
\end{itemize}
Now the lemma is an immediate consequence of the above classification and the relations in \eqref{brackets31}.
\end{proof}

\begin{remark}\label{Rem1}
By a  well-known fact, there exists a unique connected $2$-dimensional non-abelian Lie group, up to isomorphism (see \cite[p.212]{GOV}), namely $Aff_\circ(\R)$ which is the identity component of $Aff(\R)$, the group of affine transformations of the real line $\R$. The group $Aff(\R)$ consists of the translations and homotheties, therefore, it is isomorphic to the semi-direct product $\R^*\ltimes \R$. In fact, $Aff(\R)$ acts on $\R$ as following:
\[(r,v).x:=rx+v,\quad\quad(r,v)\in Aff(\R)\simeq \R^*\ltimes \R, \;x\in \R.\]
Hence, $Aff_\circ(\R)$ is isomorphic to $\R^*_+\ltimes \R$. This group acts on $\R$ transitively. Hence, the stabilizer of each point of $\R$ is conjugate to the stabilizer of the origin, which is $\R^*_+$.

 We show that $\R^*_+\ltimes \R$ has exactly two distinct $1$-parameter subgroups, up to conjugacy, namely $\R^*_+$ and $\R$. Let $\R\oplus_\theta \R$ be the Lie algebra of $\R^*_+\ltimes \R$, where $\theta:\R\rightarrow \fg \mathfrak{l}(\R)$ is given by  $\theta (\lambda)(x)=\lambda x$. Let $\omega=(\lambda,u)$ be an arbitrary nonzero element of $\R\oplus_\theta \R$.  Note that, if $\lambda=0$ (resp. $u=0$), then $\exp(\R \omega)=\R$ (resp. $\exp(\R \omega)=\R_+^*$). Now, suppose that $\lambda u\neq 0$. Without losing generality assume that $\lambda=1$. Then it can be easily seen that $G_{-u}=\{(e^t,(e^t-1)u):t\in \R\}=\exp(\R \omega)$. This shows that $\exp (\R\omega)$ is conjugate to $\R^*_+$, within $Aff_\circ (\R)$.
 
 Based on the above observation, any one dimensional Lie subgroup of $AN_1$ is conjugate to either $A$ or $N_1$, within $SO_\circ (3,1)$. The conjugation matrix, as an isometry, preserves $\R e_2$. A similar fact remains true for $AN_2$. Since $N$ is isomorphic to the additive Lie group $\R^2$, any one parameter subgroup of $N$ is conjugate to $N_1$. In this case, the conjugation matrix comes from $K_1$.   
 Therefore, using relations in \eqref{brackets31} shows that any two dimensional Lie subgroup of $SO_\circ (3,1)$ is conjugate to one of the subgroups $AN_1$, $K_1A$ or $N.$
\end{remark}
Let $L:O(3,1)\ltimes \R^{3,1}\rightarrow O(3,1)$ be the projection map on the first factor which is a Lie group homomorphism. Let $H$ be a Lie subgroup of $O(3,1)\ltimes \R^{3,1}$. Then $\ker L|_H$ is a normal subgroup of $H$, that is alled the {\it translation part} of $H$. The Lie subgroup $L(H)$ is called the {\it linear projection} of $H$ in $SO(3,1)$. Now we are ready to state our main theorem.

\begin{theorem}\label{main}
	Let $H$ be a connected subgroup of $\Iso(\R^{3,1})$ acting on $\R^{3,1}$ with cohomogeneity one. Then the action of $H$ is orbit-equivalent to the action of one of the following groups in Tables \ref{table1}-\ref{table4}.
	
	\begin{table}[h!]
		\centering
		\begin{tabular}{|c| c| c| } 
			\hline
			\multicolumn{3}{|c|}{\textbf{Subgroups with a hyperplane as the translation part}} \\
			\hline
			\textbf{Space-like}&\textbf{Lorentzian}&\textbf{degenerate} \\
			\hline
			$\R^{3}$ & $\R^{2,1}$ & $\mathbb{W}^{3}$\\
			\hline
		\end{tabular}
	\caption{Here $\mathbb{W}^{3}$ denotes the degenerate plane $\R e_1\oplus\R e_{2}\oplus \ell$.}
		\label{table1}
	\end{table}
\begin{table}[h!]
	\centering
	\begin{tabular}{|c| c|c| } 
		\hline
		\multicolumn{3}{|c|}{\textbf{Subgroups with a plane as the translation part}} \\
		\hline
		\textbf{space-like}&\textbf{Lorentzian}&\textbf{degenerate} \\
		\hline
		$SO_\circ(1,1)\times \R^{2}$ &$SO(2)\times \R^{1,1}$ &	$\exp(\R(Y_{\fa} + \lambda e_{1}))\ltimes \mathbb{W}^{2}$\\
			\hline
		&&$\exp(\R (Y^{1}_{\fn}+\lambda e_{4}))\ltimes \mathbb{W}^{2}$\\
			\hline
	\end{tabular}
	\caption{Here $\mathbb{W}^{2}$ denotes the degenerate plane $\R e_{2}\oplus \R (e_3-e_{4})\in \R^{3,1}$, and $\lambda\in \R$ is a fixed number.}
	\label{table2}
\end{table}
\begin{table}[h!]
	\centering
	\begin{tabular}{|c|c|c| } 
		\hline
		\multicolumn{3}{|c|}{\textbf{Subgroups with a line as the translation part}} \\
		\hline
		\textbf{space-like}&\textbf{time-like}&\textbf{light-like} \\
		\hline
	$SO_\circ(2,1)\times \R e_{1}$ & $SO(3)\times \R e_{4}$& $SO_{\circ}(2)\times SO_{\circ}(1,1)\ltimes \ell$\\
	\hline
		$AN_{2}\times \R e_{1}$ &  &$\exp(\R(Y_{\fa}+\lambda e_{2}))N_{1}\ltimes \ell$ \\
		\hline
					&  &$\exp(\R(Y^{1}_{\fn}+\lambda e_{2})+\R (Y^{2}_{\fn}+\lambda e_{1}+\mu e_{2}))\ltimes \ell$ \\
		\hline
					&  &$K_{1}N\ltimes \ell$ \\
		\hline
							&  &$\exp(\R (Y^{1}_{\fn}+\lambda e_{2})+ \R (Y^{2}_{\fn}+\lambda e_{1})+\R (aY^{1}_{\fk}+bY_{\fa}))\ltimes \ell$ \\
		\hline
	\end{tabular}
\caption{Here $\ell$ denotes the light-like line $\R (e_3-e_{4})\leq \R^{3,1}$, $\lambda,\mu,a,b\in \R$ are fixed numbers.}
	\label{table3}
\end{table}
	\begin{table}[h!]
		\centering
		\begin{tabular}{|c|c|c|c| } 
			\hline
			\multicolumn{4}{|c|}{\textbf{Subgroups with trivial translation part}} \\
			\hline
			$SO_{\circ}(3,1)=KAN$ & $ K_{1}AN$ &$\exp(\lambda Y^{1}_{\fk}+\mu Y_{\fa})N$ &$AN$\\
			\hline
		\end{tabular}
		\caption{Here $\lambda,\mu\in \R$ are fixed numbers.}
		\label{table4}
	\end{table}
\end{theorem}

\begin{proof}
	Let  $\pi_{1}:\mathfrak{so}(3,1)\oplus_{\varphi}\R^{3,1}\longrightarrow \mathfrak{so}(3,1)$ denote the projection map on the first factor, which is a Lie algebra homomorphism and so $\pi_1(\fh)$ is a subalgebra of $\mathfrak{so}(3,1)$, where $\fh$ is the Lie algebra of $H$. Also $\ker \pi_1|_\fh$ is an ideal of $\fh$, and obviously $\ker \pi_1|_\fh =\fh \cap \R^{3,1}$. Hence $\pi_1(\fh)$ preserves $\fh \cap \R^{3,1}$, i.e. $\pi_1(\fh)(\fh \cap \R^{3,1})\subseteq \fh \cap \R^{3,1}$ by \eqref{braketrule}.  We mean by the normalizer of $\fh \cap \R^{3,1}$ in $\mathfrak{so}(3,1)$ the Lie subalgebra of $\mathfrak{so}(3,1)$ preserving $\fh \cap \R^{3,1}$. We consider the proof according to the dimension of $\fh \cap \R^{3,1}$.
	
\textbf{Case I.} $\dim(\fh\cap \R^{3,1})= 0$. The groups obtained in this case fill Table \ref{table1}. Since $\dim(\fh\cap \R^{3,1})= 0$, we must have $\dim(\pi_{1}(\fh))\geq 3$. By Lemma \ref{subgroups}, $\pi_1(\fh)$ can be one of the following Lie subalgebras
 $$\mathfrak{so}(3,1),\quad  \fk_{1}\oplus \fa\oplus \fn,\quad \fs (2,1),\quad \mathfrak{so} (3),\quad \R(\lambda Y_{\fk}^1 +\mu Y_\fa)\oplus \fn,$$
where $\lambda$ and $\mu$ are fixed real numbers. The subgroups $SO_\circ(2,1)=\exp(\fs (2,1))$ and $SO(3)=\exp(\fs (3))$, from this list, do not act with cohomogeneity one obviously. We claim that if $\pi_1(\fh)$ is conjugate to one of the remaining cases of the list, then $\fh$ does.
\begin{itemize}
	\item Let $\pi_1(\fh)=\fs(3,1)$, up to conjugacy.
Then there exist $ u,v,w,x, y,z\in \R^{3,1}$ such that
$$ \fh=\R (Y_{\mathfrak{k}}^{1}+u)+\R (Y_{\mathfrak{k}}^{2}+v)+\R (Y_{\mathfrak{k}}^{3}+w)+\R(Y_{\mathfrak{a}}+x)+ \R(Y_{\fn}^{1}+y)+\R(Y_{\fn}^{2}+z).$$ 
Hence $u,v,w,x,y$ and $z$ should satisfy the relations in \eqref{vectors4}.
Let $p=(u_{2},-u_{1},-v_{1},-x_{3})^{t}$. Then $Ad((I_{4},p))$ maps $Y^{1}_{\fk}+u, Y^{2}_{\fk}+v, Y^{3}_{\fk}+w, Y_{\fa}+x, Y^{1}_{\fn}+y$ and $Y^{2}_{\fn}+z$ to $Y^{1}_{\fk}, Y^{2}_{\fk}, Y^{3}_{\fk}, Y_{\fa}, Y^{1}_{\fn}$ and $Y^{2}_{\fn}$, respectively. Hence $Ad((I_{4},p))(\fh)=\mathfrak{so}(3,1)$. Therefore the action of $H$ is orbit-equivalent to the action of $SO_{\circ}(3,1)$.

\item Let $\pi_1(\fh)=\fk_1\oplus \fa\oplus \fn$. Then $\fh =\R(Y^{1}_{\fk}+u)+\R(Y_{\fa}+x)+\R(Y^{1}_{n}+y)+\R(Y^{2}_{n}+z)$. By using the relations in \eqref{Bracketv31} and \eqref{vectors31} one gets that
\begin{align*}
&u=(u_{1},u_{2},0,0)^{t},\qquad\qquad\ x=(0,0,x_{3},x_{4})^{t},\\ &y=(x_{3}+x_{4}, 0,u_{2},-u_{2})^{t},\quad z=(0,x_{3}+x_{4},-u_{1}, u_{1})^{t}.
\end{align*}
Let $p=(u_{2},-u_{1},-x_{4},-x_{3})^{t}$. Then $Ad((I_{4},p))$ maps $Y^{1}_{\fk}+u, Y_{\fa}+x, Y^{1}_{n}+y$ and $Y^{2}_{n}+z$ to $Y^{1}_{\fk}, Y_{\fa}, Y^{1}_{n}$ and $Y^{2}_{n}$, respectively. Hence $Ad((I_{4},p))(\fh)=\fk_{1}\oplus \fa\oplus \fn$. Therefore the action of $H$ is orbit-equivalent to the action of $K_{1}AN$.

\item Let $\pi_1(\fh)=\R(\lambda Y_{\fk}^1 +\mu Y_\fa)\oplus \fn$, where $\lambda$ and $\mu$ are fixed real numbers. So $\fh =\R(\lambda Y^{1}_{\fk}+\mu Y_{\fa}+x)+\R(Y^{1}_{\fn}+y)+\R(Y^{2}_{\fn}+z)$, for some $x,y,z\in\R^{3,1}$. First assume that $\lambda\mu\neq 0$. The relations in \eqref{Bracketv31} and \eqref{vectors31} imply that
$x=(x_{1},x_{2},x_{3},x_{4})^{t}, y=(\frac{1}{\mu}(x_{3}+x_{4}),0,\frac{1}{\lambda}x_{2},-\frac{1}{\lambda}x_{2})^{t}, z=(0,\frac{1}{\mu}(x_{3}+x_{4}),-\frac{1}{\lambda}x_{1},\frac{1}{\lambda}x_{1})^{t}$. Let $p=(\frac{1}{\lambda}x_{2},-\frac{1}{\lambda}x_{1},-\frac{1}{\mu}x_{4},-\frac{1}{\mu}x_{3})^{t}$. Then $ Ad((I_{4},p))$ maps $\lambda Y^{1}_{\fk}+\mu Y_{\fa}+x, Y^{1}_{\fn}+y$ and $Y^{2}_{\fn}+z$ to $\lambda Y^{1}_{\fk}+\mu Y_{\fa}, Y^{1}_{\fn}$ and $Y^{2}_{\fn}$, respectively. Hence $Ad((I_{4},p))(\fh)=\R(\lambda Y^{1}_{\fk}+\mu Y_{\fa})\oplus \fn$. Therefore the action of $H$ is orbit-equivalent to the action of $\exp(\R(\lambda Y^{1}_{\fk}+\mu Y_{\fa}))N$. \\
Now assume that $\lambda=0$, and so $\mu\neq 0$. 
Then  $\fh =\R(Y_{\fa}+x)+\R(Y^{1}_{\fn}+y)+\R(Y^{2}_{\fn}+z)$. The relations in \eqref{Bracketv31} and \eqref{vectors31} show that
$x=(0,0,x_{3},x_{4})^{t}, y=(x_{3}+x_{4},0,y_{3},-y_{3})^{t}, z=(0,x_{3}+x_{4},z_{3},-z_{3})^{t}$. Let $p=(y_{3},z_{3},-x_{4},-x_{3})^{t}$. Then $ Ad((I_{4},p))$ maps $Y_{\fa}+x, Y^{1}_{\fn}+y$ and $Y^{2}_{n}+z$ to $Y_{\fa}, Y^{1}_{\fn}$ and $Y^{2}_{\fn}$, respectively. Hence $Ad((I_{4},p))(\fh)=\fa\oplus \fn$. Therefore the action of $H$ is orbit-equivalent to the action of $AN$.\\
The case that $\mu=0$ and $\lambda\neq 0$ is excluded automatically, since the action of $K_1N$ is not of cohomogeneity one. (See \cite{AS}).
\end{itemize}

\textbf{Case II.} $\dim(\fh \cap \R^{3,1})=1$. The action is of cohomogeneity one, so $\dim(\pi_{1}(\fh))\geq 2$. Every one-dimensional space-like, time-like or light-like subspace of $\R^{3,1}$ is conjugate under $O(3,1)$ to $\R e_{1}, \R e_{4}$ or $\ell$, respectively. We can therefore assume that $\fh\cap \R^{3,1}$
is equal to one of these three one-dimensional subspaces.

\textbf{Case II-1.}   $\fh\cap \R^{3,1}=\R e_{1}$. The normalizer of $\R e_{1}$ in $\mathfrak{so}(3,1)= \fk\oplus \fa\oplus \fn$ is equal
to $\fk_{3}\oplus \fa\oplus \fn_{2}$, which implies that $\pi_{1}(\fh)\subseteq \fk_{3}\oplus \fa\oplus \fn_{2}$. 

If $\pi_1(\fh)=\fk_{3}\oplus \fa\oplus \fn_{2}$, then $\fh$ is of the form $ \R(Y^{3}_{\fk}+u)\oplus \R(Y_{\fa}+v)\oplus \R(Y^{2}_{\fn}+w)\oplus \R e_{1}$, where $u,v,w\in \R^{3,1}$. By using the relations in \eqref{Bracketv31} and \eqref{vectors31} we have
$$ u=(0,u_{2},u_{3},0)^{t},v=(0,0,v_{3},u_{2})^{t},w=(0,v_{3}+u_{2},u_{3},-u_{3})^{t}.$$
Let $p=(0,u_{3}, -u_{2}, -v_{3})^{t}$. Then $Ad((I_{4},p))$ maps $Y^{3}_{\fk}+u, Y_{\fa}+v$ and $ Y^{2}_{\fn}+w$ to $Y^{3}_{\fk}, Y_{\fa}$ and $ Y^{2}_{\fn}$, respectively. Hence $Ad((I_{4},p))(\fh) =\R Y^{3}_{\fk}+\R Y_{\fa}+\R Y^{2}_{\fn}\oplus \R e_{1}$. Thus the action of $H$ is orbit-equivalent to the action of 
$SO_\circ(2,1)\times \R e_{1}$, and the action of this group is obviously of \co.

If $\pi_1(\fh)\subsetneqq\fk_{3}\oplus \fa\oplus \fn_{2}=\mathfrak{so}(2,1)$, (we remind that $\dim \pi_1(\fh)\geqslant 2$), then a well-known fact about two dimensional subgroups of $SO_\circ (2,1)$ says that $\exp(\pi_1(\fh))$  is conjugate to $AN_2$. Hence $\fh$, as a vector space, is of the form
 $\R(Y_{\fa}+u)+\R(Y^{2}_{\fn}+v)\oplus \R e_{1}$. 
 The relations in \eqref{Bracketv31} and \eqref{vectors31} show that $u=(u_{1},0,u_{3},u_{4})^{t}$ and  $v=(0,u_{3}+u_{4},v_{3},-v_{3})^{t}$. Let $p=(0,v_{3}, -u_{4}, -u_{3})^{t}$. Then $Ad((I_{4},p))$ maps $Y_{\fa}+u$ and $Y^{2}_{\fn}+v$ to  $Y_{\fa}+u_{1} e_{1}$ and $Y^{2}_{\fn}$, respectively. Hence $Ad((I_{4},p))(\fh) =\R(Y_{\fa}+\lambda e_{1})+\R Y^{2}_{\fn}\oplus \R e_{1}$, where $\lambda$ is a fixed real number. Thus the action of $H$ is orbit-equivalent to the action of 
$ AN_{2}\times \R e_{1}$.

\textbf{Case II-2.} $\fh\cap \R^{3,1}=\R e_{4}$. The normalizer of $\R e_{4}$ in $\mathfrak{so}(3,1)= \fk\oplus \fa\oplus \fn$ is equal
to $\fk$, which implies that $\pi_{1}(\fh)\subseteq \fk=\fs (3)$. If $\fh\subsetneqq \fs (3)$, then by the fact that $\fs (3)$ has no $2$-dimensional subalgebra, $\dim (\fh)\leqslant 2$, and so $H$ can not act with cohomogeneity one on $\R^{3,1}$.
 If $\fh=\fs (3)$, then $\fh$ is of the form $(\R(Y^{1}_{\fk}+u)+\R(Y^{2}_{\fk}+v)+\R(Y^{3}_{\fk}+w))\oplus \R e_{4}$ with $u,v,w\in \R^{3,1}$. Hence the relations in \eqref{Bracketv31} and \eqref{vectors31} show that $u=(u_{1},u_{2},0,0)^{t},v=(v_{1},0,u_{2},0)^{t}$ and $w=(0,v_{1},-u_{1},0)^{t}$. Let $p=(u_{2}, -u_{1},-v_{1},0)^{t}$. Then  $Ad((I_{4},p))$ maps $Y^{1}_{\fk}+u, Y^{2}_{\fk}+v$ and $Y^{3}_{\fk}+w$ to  $Y^{1}_{\fk}, Y^{2}_{\fk}$ and $Y^{3}_{\fk}$, respectively. Hence
$Ad((I_{4},p))(\fh) = \fk\oplus \R e_{4}$. Thus the action of $H$ is orbit-equivalent to the action of $SO(3)\times \R e_{4}$, where its action on $\R^{3,1}$ is obviously of \co.

\textbf{Case II-3.} Assume that $\fh\cap \R^{3,1}= \ell$. The normalizer of $\ell$ in $\mathfrak{so}(3,1)= \fk\oplus \fa\oplus \fn$ is equal
to $\fk_{1}\oplus \fa\oplus \fn$, which implies thet $\pi_{1}(\fh)\subseteq \fk_{1}\oplus \fa\oplus \fn$. 

We claim that the case $\pi_{1}(\fh)=\fk_{1}\oplus \fa\oplus \fn$ does not imply cohomogeneity one action. If $\pi_{1}(\fh)=\fk_{1}\oplus \fa\oplus \fn$, then $\fh$ is of the form $\R (Y^{1}_{\fk}+u)+\R (Y_{\fa}+v)+\R (Y^{1}_{\fn}+w)+\R (Y^{2}_{\fn}+x)+\ell$ and so the same argument of that of the second item of Case I shows that the action of $H$ is orbit equivalent to the action of $K_{1}AN\ltimes \ell$. We claim that the action of this group is not of \co. Let $q=(q_1,q_2,q_3,q_4)^t\in \R^{3,1}$. If $q_{3}+q_{4}\neq 0$, then $\dim H(q)=4$, if $q_{3}+q_{4}=0$ and $q$ is not the origin, then $\dim H(q)=2$ and if $q$ is the origin $\dim H(q)=1$. Hence $H$ does not act with cohomogeneity one on $\R^{3,1}$. Thus the case $\fh\cap \R^{3,1}= \ell$, where $\pi_{1}(\fh)=\fk_{1}\oplus \fa\oplus \fn$, is excluded.

Therefore, $\pi_{1}(\fh)\subsetneqq\fk_{1}\oplus \fa\oplus \fn$, and so $\dim\pi_1(\fh)\in\{2,3\}$. 

\textbf{Subcase II-3-a.} Let $\dim(\pi_{1}(\fh))= 2$. Considering Remark \ref{Rem1}, $\fh$ as a vector space can be one of the following subspaces:

(i) $\R(Y^{1}_{\fk}+u)+\R(Y_{\fa}+v)\oplus \ell$,

(ii) $\R(Y_{\fa}+u)+\R(Y^{1}_{\fn}+v)\oplus \ell$,

(iii) $\R(Y^{1}_{\fn}+u)+\R(Y^{2}_{\fn}+v)\oplus \ell$.


In case (i), the relations in \eqref{Bracketv31} and \eqref{vectors31} show that $u=(u_{1},u_{2},0,0)^{t}$ and $v=(0,0,v_{3},v_{4})^{t}$. Let $p=(u_{2},-u_{1}, -v_{4}, -v_{3})^{t}$. Then $Ad((I_{4},p))$ maps $Y^{1}_{\fk}+u$ and $Y_{\fa}+v$ to $Y^{1}_{\fk}$ and $Y_{\fa}$, respectively. Hence $Ad((I_{4},p))(\fh) =\fk_{1}\oplus \fa\oplus \ell$. Therefore the action of $H$ is orbit-equivalent to the action of $K_1A\ltimes \ell=SO(2)\times SO_\circ(1,1)\ltimes \ell$. For any point  $q=(q_1,q_2,q_3,q_4)^t\in\R^{3,1}$, where $q_{2}q_{3}\neq 0$,  $\dim H(q)=3$ and so $H$ acts with cohomogeneity one on $\R^{3,1}$.

In case (ii), by using the relations in \eqref{Bracketv31} and \eqref{vectors31} one gets that $u=(0,u_{2},u_{3},u_{4})^{t}$ and $v=(u_{3}+u_{4},0,v_{3},-v_{3})^{t}$. Let $p = (v_{3}, 0, -u_{4}, -u_{3})^{t}$. Then $Ad((I_{4},p))$ maps $Y_{\fa}+u$ and $Y^{1}_{\fn}+v$ to $Y_{\fa}+u_{2}e_{2}$ and $Y^{1}_{\fn}$, respectively. Hence $Ad((I_{4},p))(\fh)=\R(Y_{\fa}+\lambda e_{2})+\R Y^{1}_{\fn}\oplus \ell$, where $\lambda$ is a fixed real number. Thus the action of $H$ is orbit-equivalent to the action of $\exp(\R(Y_{\fa}+\lambda e_{2}))N_{1}\ltimes \ell$. For any point $q=(q_1,q_2,q_3,q_4)^t\in\R^{3,1}$, where $q_{3}+q_{4}\neq 0$,  $\dim H(q)=3$ and so $H$ acts with cohomogeneity one on $\R^{3,1}$.


In case (iii), the relations in \eqref{Bracketv31} and \eqref{vectors31} show that $u=(u_{1},u_{2},u_{3},-u_{3})^{t}$ and $v=(u_{2},v_{2},v_{3},-v_{3})^{t}$, respectively. Let $p = (u_{3},v_{3},-u_{1},0)^{t}$. Then $Ad((I_{4},p))$ maps $(Y^{1}_{\fn}+u)$ and $Y^{2}_{\fn}+v$ to $Y^{1}_{\fn}+u_{2}e_{2}$ and $Y^{2}_{\fn}+u_{2}e_{1}+(v_{2}-u_{1})e_{2}$, respectively. Hence $Ad((I_{4},p))(\fh)=\R(Y^{1}_{\fn}+\lambda e_{2})+\R (Y^{2}_{\fn}+\lambda e_{1}+\mu e_{2})\oplus \ell$, where $\lambda$ and $\mu$ are fixed real numbers.Thus the action of $H$ is orbit-equivalent to the action of $\exp(\R(Y^{1}_{\fn}+\lambda e_{2})+\R (Y^{2}_{\fn}+\lambda e_{1}+\mu e_{2}))\ltimes \ell$. 
 Let $q=(q_1,q_2,q_3,q_4)^t$ be an arbitrary point of $\R^{3,1}$. If $\lambda\neq 0$, then for $q_{3}+q_{4}=0$, we have $\dim H(q)=3$ and if $\lambda=0$, then for $q_{3}+q_{4}\neq -\mu$, we get $\dim H(q)=3$. Thus $H$ acts with cohomogeneity one on $\R^{3,1}$.

%
 
\textbf{Subcase II-3-b.}  Let $\dim(\pi_{1}(\fh))= 3$. By Lemma \ref{subgroups}, any three dimensional Lie subalgebra of $\fk_{1}\oplus\fa\oplus\fn$ is of the form $\R(\lambda Y^{1}_{\fk}+\mu Y_{\fa})\oplus \fn$. Hence $\fh$, as a vector space, is of the form
$\fh=\R(a Y^{1}_{\fk}+b Y_{\fa}+w)+\R(Y^{1}_{\fn}+u)+\R(Y^{2}_{\fn}+v)\oplus \ell$,
where $a,b\in \R$, $u,v,w\in \R^{3,1}$ and at least one of $a$ or $b$ is not zero.

The relations in \eqref{brackets31} show that $[Y^{1}_{\fn},aY^{1}_{\fk}+bY_{\fa}]=aY^{2}_{\fn}+bY^{1}_{\fn}$, and so there are fixed real numbers $a_{1},a_{2}$ and $a_{3}$ such that
$aY^{2}_{\fn}+bY^{1}_{\fn}=a_{1}Y^{1}_{\fn}+a_{2}Y^{2}_{\fn}+a_{3}(aY^{1}_{\fk}+bY_{\fa})$, which implies that $a_{1}=b, a_{2}=a, a_{3}a=0$ and $a_{3}b=0$.
\begin{itemize}
\item If $ab\neq 0$, then $\fh=\R(Y^{1}_{\fn}+u)+\R(Y^{2}_{\fn}+v)+\R(aY^{1}_{\fk}+bY_{\fa}+w)\oplus \ell$. The relations in \eqref{Bracketv31} and \eqref{vectors31} imply that $u=(u_{1},u_{2},u_{3},-u_{3})^{t}, v=(0,u_{1},v_{3},-v_{3})^{t}$ and $w=(-av_{3},au_{3},w_{3},bu_{1}-w_{3})^{t}$. Let $p = (u_{3},v_{3},-u_{1},0)^{t}$. Then $Ad((I_{4},p))$ maps $Y^{1}_{\fn}+u, Y^{2}_{\fn}+v$ and $aY^{1}_{\fk}+bY_{\fa}+w$ to $Y^{1}_{\fn}+u_{2}e_{2}, Y^{2}_{\fn}+u_{2}e_{1}$ and $aY^{1}_{\fk}+bY_{\fa}+w_{3}(e_{3}-e_{4})$, respectively. Hence $Ad((I_{4},p))(\fh)=\R (Y^{1}_{\fn}+\lambda e_{2})+ \R (Y^{2}_{\fn}+\lambda e_{1})+\R (aY^{1}_{\fk}+bY_{\fa}+\mu(e_{3}-e_{4}))\oplus \ell$, where $\lambda$ and $\mu$ are fixed real numbers. Thus the action of $H$ is orbit-equivalent to the action of $\exp(\R (Y^{1}_{\fn}+\lambda e_{2})+ \R (Y^{2}_{\fn}+\lambda e_{1})+\R (aY^{1}_{\fk}+bY_{\fa}+\mu(e_{3}-e_{4})))\ltimes \ell$ and its action is orbit-equivalent to the action of $\exp(\R (Y^{1}_{\fn}+\lambda e_{2})+ \R (Y^{2}_{\fn}+\lambda e_{1})+\R (aY^{1}_{\fk}+bY_{\fa}))\ltimes \ell$. For any point $q\in\R^{3,1}$, where $q_{3}+q_{4}\neq \pm \lambda$, we have $\dim(H(p))=3$ and so $H$ acts with cohomogeneity one on $\R^{3,1}$.

\item If $a=0$, then we may assume that $b=1$, and so $\fh=\R(Y^{1}_{\fn}+u)+\R(Y^{2}_{\fn}+v)+\R(Y_{\fa}+w)\oplus \ell$. Hence by using the third part of the second item of Case I one gets that $\fh$ is conjugate to $\fa\oplus \fn\oplus \ell$. Thus the action of $H$ is orbit-equivalent to the action of $AN\ltimes \ell$. We claim that the action of this group is not of \co. Let $q=(q_1,q_2,q_3,q_4)^t\in \R^{3,1}$. If $q_{3}+q_{4}\neq 0$, then $\dim H(q)=4$, if $q_{3}+q_{4}=0$, then $\dim H(q)=1$. Thus $H$ does not act with cohomogeneity one on $\R^{3,1}$. Therefore, the case that $a=0$ is excluded.
		 
\item If $b=0$, then we may assume that $a=1$, and so without lossing the generality  $\fh=\R(Y^{1}_{\fk}+u)+\R(Y^{1}_{\fn}+v)+\R(Y^{2}_{\fn}+w)\oplus \ell$. By using the relations in \eqref{Bracketv31} and \eqref{vectors31} one gets that $u=(u_{1},u_{2},u_{3},-u_{3})^{t}, v=(v_{1},0,u_{2},-u_{2})^{t}$ and $w=(0,v_{1},-u_{1},u_{1})^{t}$. Let $p = (u_{2},-u_{1},-v_{1},0)^{t}$. Then $Ad((I_{4},p))$ maps $Y^{1}_{\fk}+u, Y^{1}_{\fn}+v$ and $Y^{2}_{\fn}+w$ to $Y^{1}_{\fk}+u_{3}(e_{3}-e_{4}), Y^{1}_{\fn}$ and $Y^{2}_{\fn}$, respectively. Hence $Ad((I_{4},p))(\fh)=\R (Y^{1}_{\fk}+\lambda(e_{3}-e_{4}))+ \R Y^{1}_{\fn}+\R Y^{2}_{\fn}\oplus \ell$, where $\lambda$ is a fixed real number. Thus the action of $H$ is orbit-equivalent to the action of 
$K_{1}N\ltimes \ell$. For any point $q=(q_1,q_2,q_3,q_4)^t\in\R^{3,1}$, where $q_{3}+q_{4}\neq 0$,  $\dim H(q)=3$ and so $H$ acts with cohomogeneity one on $\R^{3,1}$.
\end{itemize}

\textbf{Case III.} $\dim(\fh \cap \R^{3,1})=2$. Every two-dimensional space-like, time-like or light-like subspace of $\R^{3,1}$ is conjugate under $O(3,1)$ to $\R e_{1}\oplus \R e_{2}, \R e_{3}\oplus \R e_{4} $ or $\R e_{2}\oplus \ell$, respectively. We can therefore assume that $\fh\cap \R^{3,1}$  is equal to one of these three two-dimensional subspaces.

\textbf{Case III-1.}  $\fh\cap \R^{3,1}=\R e_{1}\oplus \R e_{2}$. The normalizer of $\R e_{1}\oplus \R e_{2}$ in $\mathfrak{so}(3,1)= \fk\oplus \fa\oplus \fn$ is equal to $\fa$, which implies that $\pi_{1}(\fh)\subseteq\fa$. Hence $\fh$ is of the form $\R(Y_{\fa}+u)\oplus (\R e_{1}\oplus \R e_{2})$, where  $u\in \R^{3,1}$. Let $p = (0,0,-u_{4},-u_{3})^{t}$, then $Ad((I_{4},p))$ maps $Y_{\fa}+u$ to $Y_{\fa}+ u_{1} e_{1}+u_{2} e_{2}$. Hence $Ad((I_{4},p))(\fh) = \R(Y_{\fa}+\lambda e_{1} +\mu e_{2}))\oplus (\R e_{1}\oplus \R e_{2})$, where $\lambda$ and $\mu$ are fixed real numbers. Thus the action of $H$ is orbit-equivalent to the action of $\exp(\R(Y_{\fa}+\lambda e_{1} +\mu e_{2}))\ltimes (\R e_{1}\oplus\R e_{2})$,  and its action is orbit-equivalent to the action of $SO_\circ(1,1)\times \R^2$. The action of this group is clearly with \co.

\textbf{Case III-2.} $\fh\cap \R^{3,1}= \R e_{3}\oplus \R e_{4}$. The normalizer of $\R e_{3}\oplus \R e_{4}$ in $\mathfrak{so}(3,1)= \fk\oplus \fa\oplus \fn$ is equal to $\fk_{1}$, which implies that $\pi_{1}(\fh)\subseteq\fk_{1}$. Hence $\fh$ is of the form $\R(Y^{1}_{\fk}+u)\oplus (\R e_{3}\oplus \R e_{4})$, where $u\in \R^{3,1}$. Let $p = (u_{2},-u_{1},0,0)^{t}$, then $Ad((I_{4},p))$ maps $Y_{\fk}^{1}+u$ to $Y_{\fk}^{1}+ u_{3}e_{3}+ u_{4}e_{4}$. Hence $Ad((I_{4},p))(\fh)=\R(Y_{\fk}^{1} + \lambda e_{3}+\mu e_{4})\oplus (\R e_{3}\oplus \R e_{4})$, where $\lambda$ and $\mu$ are fixed real numbers. Thus the action of $H$ is orbit-equivalent to the action of $\exp(\R(Y_{\fk}^{1} + \lambda e_{3}+\mu e_{4}))\ltimes (\R e_{3}\oplus \R e_{4})$, and its action is orbit-equivalent to the action of $SO(2)\times \R^{1,1} $. Obviously, the action of this group on $\R^{3,1}$ is of \co.

\textbf{Case III-3.} $\fh\cap \R^{3,1}= \R e_{2}\oplus \ell$. The normalizer of $\R e_{2}\oplus \ell$ in $\mathfrak{so}(3,1)= \fk\oplus \fa\oplus \fn$ is equal to $\fa\oplus \fn_{1}$, which implies that $\pi_{1}(\fh)\subseteq\fa\oplus \fn_{1}$. 

We claim that the case $\pi_{1}(\fh)=\fa\oplus \fn_{1}$ does not lead to a cohomogeneity one action of $H$ on $\R^{3,1}$. If $\pi_{1}(\fh)=\fa\oplus \fn_{1}$, then  $\fh=\R(Y_{\fa}+u)+\R(Y^{1}_{\fn}+v)\oplus (\R e_{2}\oplus \ell)$, where  $u,v\in \R^{3,1}$. Let $p=(v_{3},0,-u_{4},-u_{3})^{t}$. Then $Ad((I_{4},p))$ maps $Y_{\fa}+u$ and $Y^{1}_{\fn}+v$ to $Y_{\fa}+u_{2} e_{2}$ and $Y^{1}_{\fn}$, respectively. Hence $Ad((I_{4},p))(\fh) = \R (Y_{\fa}+\lambda e_{2})+\R Y^{1}_{\fn}\oplus (\R e_{2}\oplus \ell)$, where $\lambda$ is a fixed real number. Hence the action of $H$ is orbit-equivalent to the action of $\exp(\R (Y_{\fa}+\lambda e_{2})+\R Y^{1}_{\fn})\ltimes \mathbb{W}^{2}$, which its action is orbit-equivalent to the action of $AN_{1}\ltimes \mathbb{W}^{2}$. We show that the action of this group on $\R^{3,1}$ is not of \co. 
 Let $q=(q_1,q_2,q_3,q_4)^t$ be an arbitrary point of $\R^{3,1}$. If $q_{3}+q_{4}\neq 0$, then $\dim H(q)=4$, and if $p_{3}+p_{4}=0$, then $\dim H(q)=2$. Thus $H$ does not act with cohomogeneity one on $\R^{3,1}$, and so the case $\fh\cap \R^{3,1}= \R e_{2}\oplus \ell$, with 
 $\pi_{1}(\fh)=\fa\oplus \fn_{1}$, is excluded.

If $\pi_{1}(\fh)\subsetneq\fa\oplus \fn_{1}$, then $\fh$ is of the form
$\R(aY_{\fa}+bY^{1}_{\fn}+u)\oplus (\R e_{2}\oplus \ell)$, where $a,b\in \R$, $u\in \R^{3,1}$ and at least one of $a$ or $b$ is not zero. 
\begin{itemize}
\item If $a\neq 0$, then $\fh=\R(Y_{\fa}+\frac{b}{a}Y^{1}_{\fn}+u)\oplus (\R e_{2}\oplus \ell)$. Therefore from Remark \ref{Rem1} there exist $(J,p_{0})\in \Iso(\R^{3,1})$, preserving $\mathbb{W}^2$, such that $Ad((J,p_{0}))$ maps $Y_{\fa}+\frac{b}{a}Y^{1}_{\fn}+u$ to $Y_{\fa}+u^{\prime}$, and we also denote $u^{\prime}$ by $u$. let $p = (0,0,-u_{4},-u_{3})^{t}$. Then $Ad((I_{4},p))$ maps $Y_{\fa}+u$ to $Y_{\fa}+u_{1}e_{1}+u_{2}e_{2}$. Hence $Ad((I_{4},p))(\fh) = \R(Y_{\fa}+ \lambda e_{1}+\mu e_{2})\oplus \mathbb{W}^{2}$, where $\lambda$ and $\mu$ are fixed real numbers. Thus the action of $H$ is orbit-equivalent to the action of 
$\exp(\R(Y_{\fa}+ \lambda e_{1}))\ltimes \mathbb{W}^{2}$. For any point $q=(q_1,q_2,q_3,q_4)^t\in\R^{3,1}$, where, $q_{3}+q_{4}\neq 0$, we have $\dim H(q)=3$ and so $H$ acts with cohomogeneity one on $\R^{3,1}$.


\item If $a=0$, then we may assume that $b=1$, and so $\fh=\R (Y^{1}_{\fn}+u)\oplus (\R e_{2}\oplus \ell)$. Let $p=(u_{3},0,-u_{1},0)^{t}$. Then $Ad((I_{4},p))$ maps $Y^{1}_{\fn}+u$ to $Y^{1}_{\fn}+u_{2}e_{2}+(u_{3}+u_{4})e_{4}$. Hence $Ad((I_{4},p))(\fh) = \R (Y^{1}_{\fn}+ \mu e_{2}+\lambda e_{4})\oplus (\R e_{2}\oplus \ell)$, where $\lambda$ and $\mu$ are fixed real numbers. Hence the action of $H$ is orbit-equivalent to the action of 
$\exp(\R (Y^{1}_{\fn}+\lambda e_{4}))\ltimes \mathbb{W}^{2}$. This group acts with cohomogeneity one on $\R^{3,1}$, since we have $\dim H(q)=3$ for any point $q\in\R^{3,1}$, where $q_3+q_4\neq 0$.
\end{itemize}

\textbf{Case IV:} $\dim(\mathfrak{h}\cap \R^{3,1})=3$. Every three-dimensional Riemannian, Lorentzian and degenerate subspace of $\R^{3,1}$ is conjugate under $O(3,1)$ to $\R^{3}, \R^{2,1}$ and $\mathbb{W}^{3}$, respectively. For dimension reasons it follows that the action of $H$ is orbit-equivalent to the action of one of the three translation subgroups $\R^{3}, \R^{2,1}$ or $\mathbb{W}^{3}$.
\end{proof}
%
%
\section{Proper and nonproper actions}
In this section we determine the proper and nonproper cohomogeneity one  actions on $\R^{3,1}$, which induced by the Lie subgroups listed in Tables \ref{table1}-\ref{table4}.

\begin{theorem}\label{nonproper}
	Let $H$ be a connected Lie subgroup of $\Iso(\R^{3,1})$, which acts isometrically and with cohomogeneity one on $\R^{3,1}$. Then the action is nonproper if and only if $H$ is conjugate to one of the following Lie groups in $O(3,1)\ltimes \R^{3,1}$.
	
	(a) $SO_\circ(1,1)\times \R^{2}$.
	
	(b) $\exp(\R(Y_{\fa}+\lambda e_{1}))\ltimes \mathbb{W}^{2}$, where $\lambda=0$.
	
	(c) $\exp(\R(Y^{1}_{\fn}+\mu e_{4}))\ltimes \mathbb{W}^{2}$, where $\mu=0$.
	
	(d) $SO_\circ(2,1)\times \R e_{1}$. 
	
	(e) $AN_{2}\times \R e_{1}$.
	
	(f) $SO_{\circ}(2)\times SO_{\circ}(1,1)\ltimes \ell$.
	
	(g) $\exp(\R(Y_{\fa}+\lambda e_{2}))N_{1}\ltimes \ell$.
	
	(h) $K_{1}N\ltimes \ell$.
	
	(i) $\exp(\R (Y^{1}_{\fn}+\lambda e_{2})+ \R (Y^{2}_{\fn}+\lambda e_{1})+\R (aY^{1}_{\fk}+bY_{\fa}))\ltimes \ell$.
	
	(j) $SO_{\circ}(3,1)=KAN$.
	
	(h) $ K_{1}AN$.
	
	(k) $\exp(aY^{1}_{\fk}+bY_{\fa})N$.
	
	(l) $AN$.
	
	(m) $\exp(\R(Y^{1}_{\fn}+\lambda e_{2})+\R (Y^{2}_{\fn}+\lambda e_{1}+\mu e_{2}))\ltimes \ell$.
\end{theorem}
\begin{proof}
 We prove that the action of any of the mentioned Lie groups is nonproper, then the proof will be a consequence of Theorem \ref{main} and the proof of Theorem \ref{proper}. All the Lie groups in cases (a) to (l), which belong to Tables (2), (3) or (4) cause a nonproper action, since in each case $H$ has a noncompact closed Lie subgroup preserving the origin. So to complete the proof we show that the action of $H=\exp(\R(Y^{1}_{\fn}+\lambda e_{2})+\R (Y^{2}_{\fn}+\lambda e_{1}+\mu e_{2}))\ltimes \ell$ is nonproper. If $\lambda=0$, then the noncompact closed subgroup $\exp (\R Y_{\fn}^1)$ acts nonproperly, since it preserves the origin, and so the action of $H$ is nonproper. Let $\lambda\neq 0$. Any element of $H$ is of the form 
 $(C_{t,s},c_{t,s,v})\in SO_{\circ}(3,1)\ltimes \R^{3,1}$, where $t,s,v\in \R$ and 
 $$C_{t,s}=\left(\begin{array}{cccc} 1&0&t&t\\0&1&s&s\\-t&-s&1-\frac{t^2+s^2}{2}&-\frac{t^2+s^2}{2}\\t&s&\frac{t^2+s^2}{2}&1+\frac{t^{2}+s^{2}}{2}
 \end{array}\right)$$
 and
 $$c_{t,s,v}=\left(\begin{array}{c} \lambda s\\\lambda t+\mu s\\ v\\-v
 \end{array}\right).$$
Let $\alpha$ be a root of the equation $x^2+\mu x-\lambda^2=0$. Consider the real sequences $\{t_{n}=(\alpha+\mu)n\},\{s_{n}=-\lambda n\}$ and $\{v_{n}=\frac{t_n^2+s_n^2}{2}\alpha\}$. Let $\{X_{n}=(x_{n},y_{n},z_{n},w_{n})^{t}\}=(0,0,\alpha,0)^t$ be a fixed sequence in $\R^{3,1}$. Then the sequence $\{g_n=(C_{t_n,s_n},c_{t_n,s_n,v_n})\}$ in $H$ has no convergent subsequence, while the two sequences $\{g_{n}.X_{n}\}$ and $\{X_{n}\}$ are convergent.
\end{proof}
\begin{theorem}\label{proper}
Let $H$ be a connected Lie subgroup of $\Iso(\R^{3,1})$ which acts isometrically and with cohomogeneity one on $\R^{3,1}$. Then the action is proper if and only if $H$ is conjugate to one of the following Lie groups in $O(3,1)\ltimes \R^{3,1}$.
	
(a) A pure translation group.
	
(b)  The standard embedding of $SO(2)\times \R^{1,1}$ in $SO_\circ (3,1)\ltimes \R^{3,1}$.

(c) The standard imbedding of $SO(3)\times \R e_{4}$ in $SO_\circ (3,1)\ltimes \R^{3,1}$.

(d) $\exp(\R(Y_{\fa}+\lambda e_{1}))\ltimes \mathbb{W}^{2}$, where $\lambda$ is a fixed nonzero real number.

(e) $\exp(\R(Y^{1}_{\fn}+\mu e_{4}))\ltimes \mathbb{W}^{2}$, where $\mu$ is a fixed nonzero real number.\\
In particular, $H$ is closed in $\Iso(\R^{3,1})$.
\end{theorem}
\begin{proof}
	Considering the Lie groups listed in Theorem \ref{nonproper} and Theorem \ref{main}, it is enough to prove that the action of the Lie groups mentioned in Theorem \ref{proper} are proper. Obviously, the action of any of the Lie groups of ($a$) to ($c$) is proper, since their linear projection are compact subgroups of $SO(3,1)$.
For the case (d), by a simple computation one gets that any element of the form $(C_{t},c_{t,s,w})\in SO_\circ(3,1)\ltimes \R^{3,1}$, where $t,s,w\in \R$ and
$$C_t=\left(\begin{array}{cccc} 1&0&0&0\\0&1&0&0\\0&0&\cosh(t)&\sinh(t)\\0&0&\sinh(t)&cosh(t)
	\end{array}\right),$$
and 
$$c_{t,s,w}=\left(\begin{array}{c} \lambda t\\s\\ w\\-w
	\end{array}\right),$$
belongs to $H$.
 Let $\{t_{n}\},\{s_{n}\}$ and $\{w_{n}\}$ be three real sequences, $\{g_n=(C_{t_n},c_{t_n,s_n,w_n})\}$ a sequence in $H$, and $\{X_{n}=(x_{n},y_{n},z_{n},v_{n})^{t}\}$ a sequence in $\R^{3,1}$. Let $g_{n}.X_{n}\rightarrow Y$ and $X_{n}\rightarrow X$, when $n\rightarrow +\infty$. If $Y=(y_{1},y_{2},y_{3},y_{4})^{t}$ and $X=(x_{1},x_{2},x_{3},x_{4})^{t}$, then $s_{n}\rightarrow y_{2}-x_{2}$ and $t_{n}\rightarrow \beta$, where $\beta=\frac{y_{1}-x_{1}}{\lambda}$ and so $w_{n}\rightarrow y_{3}-\cosh(\beta)x_{3}-\sinh(\beta)x_{4}$.

Now for case (e), by a simple computation one gets that any element of $H$ is of the form $(C_{t},c_{t,s,w})\in SO_\circ(3,1)\ltimes \R^{3,1}$ where  $t,s,w\in \R$ and
$$C_{t}=\left(\begin{array}{cccc} 1&0&t&t\\0&1&0&0\\-t&0&1-\frac{t^2}{2}&-\frac{t^2}{2}\\t&0&\frac{t^2}{2}&1+
\frac{t^{2}}{2}\end{array}\right)$$
and 
$$c_{t,s,v}=\left(\begin{array}{c} \frac{\lambda t^2}{2}\\ s\\ w\\ \lambda t-w
\end{array}\right).$$
Let $\{t_{n}\},\{s_{n}\}$ and $\{w_{n}\}$ be three real sequences, $\{g_n=(C_{t_n},c_{t_n,s_n,w_n})\}$ a sequence in $H$, and $\{X_{n}=(x_{n},y_{n},z_{n},v_{n})^{t}\}$ a sequence in $\R^{3,1}$. Let $g_{n}.X_{n}\rightarrow Y$ and $X_{n}\rightarrow X$, when $n\rightarrow +\infty$. If $Y=(y_{1},y_{2},y_{3},y_{4})^{t}$ and $X=(x_{1},x_{2},x_{3},x_{4})^{t}$, then $s_{n}\rightarrow y_{2}-x_{2}$, $t_{n}\rightarrow \frac{y_{3}+y_{4}-x_{3}-x_{4}}{\lambda}$ and so $w_{n}\rightarrow y_{3}+\frac{y_{3}+y_{4}-x_{3}-x_{4}}{\lambda}x_{1}-x_{3}+(\frac{y_{3}+y_{4}-x_{3}-x_{4}}{\lambda})^2(\frac{x_{3}+x_{4}}{2})$. Thus the action of $\exp(\R(Y^{1}_{\fn}+\mu e_{4}))\ltimes \mathbb{W}^{2}$, where $\mu\neq 0$, is proper.
\end{proof}

\section{orbits and orbit spaces of proper actions}
Let $G$ be a Lie group acting properly on a connected manifold $M$. The orbits $G(x)$ and $G(y)$ have the {\it same orbit type} if 
$G_x$ and $G_y$ are conjugate in $H$. This defines an equivalence
relation among the orbits of $H$ on $M$. Denote by $[G(x)]$ the
corresponding equivalence class, which is called the {\it orbit type} of
$G(x)$. A submanifold $S$ of $M$ is called a slice at $x$ if there is a $G$-invariant open neighborhood $U$ of $G(x)$ and a smooth equivariant retraction $r:U\rightarrow G(x)$, such that $S=r^{-1}(x)$. A fundamental feature of proper actions is the
existence of slice (see \cite{PT2}), which enables one to define a partial ordering
on the set of orbit types. The partial ordering on the set of orbit types is
defined by, $[G(y)] \leq [G(x)]$ if and only if $G_x$ is conjugate
in $G$ to some subgroup of $G_y$. If $S$ is a slice at $y$, it
implies that $[G(y)] \leq [G(x)]$ for all $x\in S$. Since $M/G$
is connected, there is a largest orbit type in the set of orbit types. Each
representative of this largest orbit type is called a principal
orbit. In other words, an orbit $G(x)$ is principal if and only if
for each point $y\in M$ the stabilizer $G_x$ is conjugate to
some subgroup of $G_y$ in $G$. Other orbits are called singular.

We start by discussing cohomogeneity one proper actions on $\R^{3,1}$. For $H\subset \Iso (\R^{3,1})$ we denote by $\mathcal{F}_{H}$ the collection of orbits of the action of $H$ on $\R^{3,1}$. By Theorem \ref{proper} we know the list of connected  subgroups of $\Iso (\R^{3,1})$ acting with \co. Hence we have four types of groups.

\textit{Type $(I)$: The Lie group $H$ is a pure translation group. Then $\mathcal{F}_{H}$ is invariant under a three-dimensional translation group}.

Let $H\in \{\R^{3}, \R^{2,1}, \mathbb{W}^{3}\}$, where $\R^{3}=\R e_{1}\oplus \R e_{2}\oplus \R e_{3}$, $\R^{2,1}=\R e_{2}\oplus \R e_{3}\oplus \R e_{4}$ and $ \mathbb{W}^{3}=\R e_{1}\oplus \R e_{2}\oplus \ell$. Then  $\mathcal{F}_{H}$ is a totally geodesic foliation of $\R^{3,1}$ whose leaves consist of the affine hyperplanes in $\R^{3,1}$ that are parallel to $\R^{3}, \R^{2,1}$ and $\mathbb{W}^{3}$ respectively:
$$\mathcal{F}_{\R^{3}}= \bigcup_{t\in \R}(te_{4}+\R^{3}),\quad \mathcal{F}_{\R^{2,1}}=\bigcup_{t\in \R} (te_{1}+\R^{2,1}),\quad \mathcal{F}_{\mathbb{W}^{3}}=\bigcup_{t\in \R}(t(e_{3}+e_{4})+\mathbb{W}^{3}).$$
Every orbit is principal and the orbit space is diffeomorphic to $\R$.

\textit{Type $(II)$: $H=SO(2)\times \R^{1,1}$. Then $\mathcal{F}_{H}$ is invariant under the two-dimensional translation group $ \R^{1,1}$}.

The action of $SO(2)$ leaves the foliation $\mathcal{F}_{\R^{2}}$ invariant and on each leaf $te_{3}+s e_4+\R^{3}\in \mathcal{F}_{\R^{2}}$ the orbits consist of the single point$\{te_{3}+se_{4}\}$, where $t,s\in \R$ and $te_3+se_{4}+S^{1}(r)$, where $r\in \R_{+}$.
Thus the set of the induced orbits of $H$ consists of the plane $\R^{1,1}$ and the pseudo-hyperbolic cylinders $ S^{1}(r)\times \R^{1,1}$, where $r\in \R_{+}$:
$$\mathcal{F}_{SO(2)\times \R^{1,1}} = \bigcup_{r\geq 0}(S^{1}(r)\times \R^{1,1}).$$
The orbit of the origin is the unique singular orbit congruent to $\R^{1,1}$. The orbit space is homeomorphic to $[0,\infty)$.

\textit{Type $(III)$: $H=SO(3)\times \R e_4$. Then $\mathcal{F}_{H}$ is invariant under the one-dimensional time-like translation group $\R e_4$}.

 The action of $SO(3)$ leaves the foliation $\mathcal{F}_{\R^{3}}$ invariant. On each leaf $ te_{4}+\R^{3}\in \mathcal{F}_{\R^{3}}$ the orbits consist of the single point ${te_{4}}$ and the spheres centered at that point. The orbits of $H$ therefore
consist of the time-like subspace $\R e_{4}$ (the singular orbit) and the cylinders $ S^{2}(r)\times \R e_{4}\subset \R^{3}\times \R e_{4}$,
where $S^{2}(r)$ is the sphere of radius $r\in \R_{+}$ in $\R^{3}:$
$$\mathcal{F}_{SO(3) \times \R e_{4}}=\R e_{4}\cup_{r\in \R_{+}}(S^{2}(r)\times \R e_{4}).$$
The unique singular orbit is $H(0)=\R e_4$. The orbit space is homeomorphic to $[0,+\infty)$.

\textit{ Type $(IV)$:  $H=\exp(\R(Y_{\fa}+\lambda e_{1}))\ltimes \mathbb{W}^{2}$, where $\lambda$ is a fixed nonzero real number.	
Then $\mathcal{F}_{H}$ is invariant under the two-dimensional degenerate translation group $\mathbb{W}^{2}$. Hence there is no spacelike orbit}.

Let $p=(p_1,p_2,p_3,p_4)^t$ be an arbitrary point of $\R^{3,1}$. The Lie group $\exp(\R(Y_{\fa}))$ leaves $\mathbb{W}^2$ invariant, so if $p_3=p_4$, then $H(p)=\mathbb{W}^3$. If $p_3\neq p_4$, then $H(p)$ is a Lorentzian generalized cylinder diffeomorphic to $\R^3$. Every orbit is a principal orbit diffeomorphic to $\R^3$. The orbit space is $\R$.

\textit{ Type $(V)$:  $H=\exp(\R(Y^{1}_{\fn}+\mu e_{4}))\ltimes \mathbb{W}^{2}$, where $\mu$ is a fixed nonzero real number.}

Let $p=(p_1,p_2,p_3,p_4)^t$ be an arbitrary point of $\R^{3,1}$. The set $\{Y_\fn^1+\mu e_4, e_2, e_3-e_4\}$ is a basis for the Lie algebra $\fg$. So the tangent space $T_p H(p)$ contains the lightlike direction $\ell$. Hence $H(p)$ is not spacelike. On the other hand
$$\frac{d}{dt}|_{t=0}(\exp (t(Y_\fn^1+\mu).p))=\mu^2t^2+(p_3+p_4)^2-2p_1\mu-\mu^2.$$
Hence, if $(p_3+p_4)^2-2p_1\mu-\mu^2<0$ (resp. $\geqslant 0$) then $H(p)$ is a Lorentzian (resp. a degenerate) hypersurface. Each orbit is a principal orbit diffeomorphic to $\R^3$ and the orbit space is $\R$.
\begin{rem}
Let $H$ be a closed and connected Lie subgroup of the isometry group of the Euclidean space $\mathbb{E}^n$ which acts on $\mathbb{E}^n$ with cohomogeneity one. Then, by Theorem 3.1 of \cite{MK}, its action is orbit equivalent to the action of $SO(k)\times \R^{n-k}$, where $1\leqslant k\leqslant n$, which is similar to that of one of the types (I) to (III). However, types (IV) and (V) clarifies the differences between cohomogeneity one actions on $\mathbb{E}^4$ and $\R^{3,1}$, when the action is proper.
\end{rem}
The following proposition is an immediate consequence of this section.
\begin{pro}\label{prop}
Let $H$ be a connected Lie subgroup of $\Iso(\R^{3,1})$, which acts properly, isometrically and with cohomogeneity one on $\R^{3,1}$. Then 

(a) the orbit space is homeomorphic to either $\R$ or $[0,+\infty)$.

(b) every singular orbit (if there is any), is either a one dimensional timelike affine subspace or a two dimensional Lorentzian affine subspace of $\R^{3,1}$. In particular, there is neither spacelike nor degenerate singular orbit. 

(c)  every orbit is geodesically complete.

(d) there is a spacelike orbit if and only if the action is orbit equivalent to the action of the pure translation group $\R^3$. 
\end{pro}

As an immediate consequence of Proposition \ref{prop}, one gets that if there exists a space-like orbit, then every orbit is a space-like hyperplane congruent to $\R^3$. This result is true in the general case, for proper actions on $\R^{n,1}$, and it has been proved in \cite{A2}. In the three dimensional case, proper actions on $\R^{2,1}$, it is proved that if there is two degenerate orbits then the orbits are parallel degenerate hyperplanes (see \cite{Ahmadi}). However, the proper action of $\exp(\R(Y^{1}_{\fn}+\mu e_{4}))\ltimes \mathbb{W}^{2}$ on $\R^{3,1}$, type (V) above, shows that the similar result is not hold for proper actions on $\R^{3,1}$.

\bibliographystyle{amsplain}

\end{document}